\documentclass{amsart}
\usepackage{amsmath,amsthm,amsfonts,amssymb}

\numberwithin{equation}{section}
\theoremstyle{plain}

\newtheorem{theorem}{Theorem}

\newtheorem{definition}[theorem]{Definition}

\newtheorem{lemma}[theorem]{Lemma}

\newtheorem{proposition}[theorem]{Proposition}

\newtheorem{remark}[theorem]{Remark}

\newcommand{\one}{{{\rm 1\mkern-1.5mu}\!{\rm I}}}

\begin{document}

\title[Averaged Large Deviations for RWRE]{Averaged Large Deviations for \\Random Walk in a Random Environment}
\author{Atilla Yilmaz}
\address
{Weizmann Institute of Science\newline
\indent Department of Mathematics\newline
\indent Rehovot 76100\newline
\indent ISRAEL}
\email{atilla.yilmaz@weizmann.ac.il}
\urladdr{http://www.wisdom.weizmann.ac.il/$\sim$yilmaz/}
\date{September 19, 2008. \textit{Revised:} May 3, 2009}

\thanks{This research was supported partially by a grant from the National Science Foundation: DMS 0604380.}

\maketitle

\begin{abstract}
In his 2003 paper, Varadhan proves the averaged large deviation principle for the mean velocity of a particle taking a nearest-neighbor random walk in a uniformly elliptic i.i.d.\ environment on $\mathbb{Z}^d$ with $d\geq1$, and gives a variational formula for the corresponding rate function $I_a$. Under Sznitman's transience condition (T), we show that $I_a$ is strictly convex and analytic on a non-empty open set $\mathcal{A}$, and that the true velocity of the particle is an element (resp.\ in the boundary) of $\mathcal{A}$ when the walk is non-nestling (resp.\ nestling). We then identify the unique minimizer of Varadhan's variational formula at any velocity in $\mathcal{A}$.
\end{abstract}

\section{Introduction}

\subsection{The model}

The random motion of a particle on $\mathbb{Z}^d$ can be modeled by a discrete time Markov chain. Write $\pi(x,x+z)$ for the transition probability from $x$ to $x+z$ for each $x,z\in\mathbb{Z}^d$, and refer to $\omega_x:=(\pi(x,x+z))_{z\in\mathbb{Z}^d}$ as the ``environment" at $x$. If the environment $\omega:=(\omega_x)_{x\in\mathbb{Z}^d}$ is sampled from a probability space $(\Omega,\mathcal{B},\mathbb{P})$, then the particle is said to take a  ``random walk in a random environment" (RWRE). Here, $\mathcal{B}$ is the Borel $\sigma$-algebra corresponding to the product topology.

Let $U:=\{(z_1,\ldots,z_d)\in\mathbb{Z}^d:|z_1|+\cdots+|z_d|=1\}$. For each $z\in U$, define the shift $T_z$ on $\Omega$ by $\left(T_z\omega\right)_x=\omega_{x+z}$. Assume that $\mathbb{P}$ is stationary and ergodic under $\left(T_z\right)_{z\in U}$,
\begin{equation}\label{nearestneighbor}
\mbox{$\mathbb{P}\{\pi(0,z)=0\}=1$ unless $z\in U$ (i.e., the walk is nearest-neighbor), and}
\end{equation}
\begin{equation}\label{ellipticity}
\mbox{$\exists\,\kappa>0$ such that $\mathbb{P}\{\pi(0,z)\geq\kappa\}=1$ for every $z\in U$. (This is called uniform ellipticity.)}
\end{equation}

For any $x\in\mathbb{Z}^d$ and $\omega\in\Omega$, the Markov chain with transition probabilities given by $\omega$ induces a probability measure $P_x^\omega$ on the space of paths starting at $x$. Statements about $P_x^\omega$ that hold for $\mathbb{P}$-a.e.\ $\omega$ are referred to as ``quenched". Statements about the semi-direct product $P_x:=\mathbb{P}\times P_x^\omega$ are referred to as ``averaged". Expectations under $\mathbb{P}, P_x^\omega$ and $P_x$ are denoted by $\mathbb{E}, E_x^\omega$ and $E_x$, respectively.

Because of the extra layer of randomness in the model, the standard questions of recurrence vs.\ transience, the law of large numbers (LLN), the central limit theorem (CLT) and the large deviation principle (LDP) --- which have well known answers for classical random walk --- become hard. However, it is possible by taking the ``point of view of the particle" to treat the two layers of randomness as one: If we denote the random path of the particle by $X:=(X_n)_{n\geq0}$, then $(T_{X_n}\omega)_{n\geq0}$ is a Markov chain (referred to as ``the environment Markov chain") on $\Omega$ with transition kernel $\overline{\pi}$ given by \[\overline{\pi}(\omega,\omega'):=\sum_{z:\,T_z\omega=\omega'}\pi(0,z).\] This is a standard approach in the study of random media. See for example \cite{DeMasi}, \cite{KV}, \cite{Kozlov}, \cite{Olla} or \cite{PV}.

See \cite{Sznitman} or \cite{Zeitouni} for a general survey of results on RWRE.

\subsection{Survey of results on quenched large deviations}

Recall that a sequence $\left(Q_n\right)_{n\geq1}$ of probability measures on a topological space $\mathbb{X}$ satisfies the LDP with rate function $I:\mathbb{X}\to\mathbb{R}^+\cup\{0\}\cup\{\infty\}$ if $I$ is lower semicontinuous, not identically infinite, and for any measurable set $G$, $$-\inf_{x\in G^o}I(x)\leq\liminf_{n\to\infty}\frac{1}{n}\log Q_n(G)\leq\limsup_{n\to\infty}\frac{1}{n}\log Q_n(G)\leq-\inf_{x\in\bar{G}}I(x).$$ Here, $G^o$ denotes the interior of $G$, and $\bar{G}$ its closure. See \cite{DZ} for general background and definitions regarding large deviations.

In the case of nearest-neighbor RWRE on $\mathbb{Z}$, Greven and den Hollander \cite{GdH} assume that $\mathbb{P}$ is a product measure, and prove

\begin{theorem}[Quenched LDP]\label{qLDPgeneric}
For $\mathbb{P}$-a.e.\ $\omega$, $\left(P_o^\omega\left(\frac{X_n}{n}\in\cdot\,\right)\right)_{n\geq1}$ satisfies the LDP with a deterministic and convex rate function $I_q$.
\end{theorem}

\noindent They provide a formula for $I_q$ and show that its graph typically has flat pieces. Their proof makes use of an auxiliary branching process formed by the excursions of the walk. By a completely different technique, Comets, Gantert and Zeitouni \cite{CGZ} extend the results in \cite{GdH} to stationary and ergodic environments. Their argument involves first proving a quenched LDP for the passage times of the walk by an application of the G\"artner-Ellis theorem, and then inverting this to get the desired LDP for the mean velocity.

For $d\geq1$, the first result on quenched large deviations is given by Zerner \cite{Zerner}. He uses a subadditivity argument for certain passage times to prove Theorem \ref{qLDPgeneric} in the case of ``nestling" walks in product environments. 

\begin{definition}
RWRE is said to be non-nestling relative to a unit vector $\hat{u}\in\mathcal{S}^{d-1}$ if
\begin{equation}\label{nonmumu}
\mathrm{ess}\inf_{\mathbb{P}}\sum_{z\in U}\pi(0,z)\langle z,\hat{u}\rangle>0.
\end{equation}
It is said to be nestling if it is not non-nestling relative to any unit vector. In the latter case, the convex hull of the support of the law of $\sum_{z}\pi(0,z)z$ contains the origin.
\end{definition}

By a more direct use of the subadditive ergodic theorem, Varadhan \cite{Raghu} drops the nestling assumption and generalizes Zerner's result to stationary and ergodic environments. The drawback of these approaches is that they don't lead to any formula for the rate function. 

Rosenbluth \cite{jeffrey} takes the point of view of the particle and gives an alternative proof of Varadhan's result. Moreover, he provides a variational formula for the rate function $I_q$. Using the same techniques, we prove in \cite{YilmazQuenched} a quenched LDP for the pair empirical measure of the environment Markov chain. This implies Rosenbluth's result by an appropriate contraction. In the same work, we also propose an Ansatz for the minimizer of the variational formula for $I_q$. We then verify this Ansatz for walks on $\mathbb{Z}$ with bounded steps.

\subsection{Previous results on averaged large deviations}

In their aforementioned paper concerning RWRE on $\mathbb{Z}$, Comets et al.\ \cite{CGZ} prove also

\begin{theorem}[Averaged LDP]\label{aLDPgeneric}
$\left(P_o\left(\frac{X_n}{n}\in\cdot\,\right)\right)_{n\geq1}$ satisfies the LDP with a convex rate function $I_a$.
\end{theorem}

\noindent They establish this result for a class of environments including the i.i.d.\ case, and obtain the following variational formula for $I_a$:
$$I_a(\xi)=\inf_{\mathbb{Q}}\left\{I_q^\mathbb{Q}(\xi) + |\xi|h\left(\mathbb{Q}\left|\mathbb{P}\right.\right)\right\}.$$
Here, the infimum is over all stationary and ergodic probability measures on $\Omega$, $I_q^\mathbb{Q}(\cdot)$ denotes the rate function for the quenched LDP when the environment measure is $\mathbb{Q}$, and $h\left(\cdot\left|\cdot\right.\right)$ is specific relative entropy. Similar to the quenched picture, the graph of $I_a$ is shown to typically have flat pieces. Note that the regularity properties of $I_a$ are not studied in \cite{CGZ}.

Varadhan \cite{Raghu} considers RWRE on $\mathbb{Z}^d$, assumes that $\mathbb{P}$ is a product measure, and proves Theorem \ref{aLDPgeneric} for any $d\geq1$. He gives yet another variational formula for $I_a$. Below, we introduce some notation in order to write down this formula.

An infinite path $\left(x_i\right)_{i\leq0}$ with nearest-neighbor steps $x_{i+1}-x_i$ is said to be in $W_\infty^{\mathrm{tr}}$ if $x_o=0$ and $\lim_{i\to-\infty}|x_i|=\infty$. For any $w\in W_\infty^{\mathrm{tr}}$, let $n_o$ be the number of times $w$ visits the origin, excluding the last visit. By the transience assumption, $n_o$ is finite. For any $z\in U$, let $n_{o,z}$ be the number of times $w$ jumps to $z$ after a visit to the origin. Clearly, $\sum_{z\in U}n_{o,z}=n_o$. If the averaged walk starts from time $-\infty$ and its path $\left(X_i\right)_{i\leq0}$ up to the present is conditioned to be equal to $w$, then the probability of the next step being equal to $z$ is
\begin{equation}\label{ozgurevren}
q(w,z):=\frac{\mathbb{E}\left[\pi(0,z)\prod_{z'\in U}\pi(0,z')^{n_{o,z'}}\right]}{\mathbb{E}\left[\prod_{z'\in U}\pi(0,z')^{n_{o,z'}}\right]}
\end{equation}
by Bayes' rule. The probability measure that the averaged walk induces on $\left(X_n\right)_{n\geq0}$ conditioned on $\{\left(X_i\right)_{i\leq0}=w\}$ is denoted by $Q^w$. As usual, $E^w$ stands for expectation under $Q^w$.

Consider the map $T^*:W_\infty^{\mathrm{tr}}\to W_\infty^{\mathrm{tr}}$ that takes $\left(x_i\right)_{i\leq0}$ to $\left(x_i-x_{-1}\right)_{i\leq-1}$. Let $\mathcal{I}$ be the set of probability measures on $W_\infty^{\mathrm{tr}}$ that are invariant under $T^*$, and $\mathcal{E}$ be the set of extremal points of $\mathcal{I}$. Each $\mu\in\mathcal{I}$ (resp. $\mu\in\mathcal{E})$ corresponds to a transient process with stationary (resp. stationary and ergodic) increments, and induces a probability measure $Q_\mu$ on particle paths $\left(X_i\right)_{i\in\mathbb{Z}}$. The associated ``mean drift" is $m(\mu) := \int\left(x_o-x_{-1}\right)\mathrm{d}\mu=Q_\mu(X_1-X_o)$. Define 
\begin{equation}\label{parisolur}
Q_\mu^w(\cdot):=Q_\mu(\,\cdot\,\left|\sigma(X_i:i\leq0)\right.)(w)\quad\text{and}\quad q_\mu(w,z):=Q_\mu^w(X_1=z)
\end{equation} for $\mu$-a.e.\ $w$ and $z\in U$. Denote expectations under $Q_\mu$ and $Q_\mu^w$ by $E_\mu$ and $E_\mu^w$, respectively.

With this notation,
\begin{equation}\label{divaneasik}
I_a(\xi)=\inf_{\substack{\mu\in\mathcal{E}:\\m(\mu)=\xi}}\mathfrak{I}_a(\mu)
\end{equation} for every $\xi\neq0$, where
\begin{equation}\label{hayirsh}
\mathfrak{I}_a(\mu):=\int_{W_\infty^{\mathrm{tr}}}\left[\sum_{z\in U}q_\mu(w,z)\log\frac{q_\mu(w,z)}{q(w,z)}\right]\,\mathrm{d}\mu(w).
\end{equation}

Aside from showing that $I_a$ is convex, Varadhan analyzes the set $$\mathcal{N}:=\left\{\xi\in\mathbb{R}^d:I_a(\xi)=0\right\}$$ where the rate function $I_a$ vanishes. For non-nestling walks, $\mathcal{N}$ consists of a single point $\xi_o$ which is the LLN velocity. In the case of nestling walks, $\mathcal{N}$ is a line segment through the origin that can extend in one or both directions. Berger \cite{Berger} shows that $\mathcal{N}$ cannot extend in both directions when $d\geq5$.

Rassoul-Agha \cite{FirasLDP} generalizes Varadhan's result to a class of mixing environments, and also to some other models of random walk on $\mathbb{Z}^d$.

\subsection{Regeneration times}\label{regtimes}

Take a unit vector $\hat{u}\in\mathcal{S}^{d-1}$. Let $$\beta=\beta(\hat{u}):=\inf\left\{k\geq0:\langle X_k,\hat{u}\rangle<\langle X_o,\hat{u}\rangle\right\}.$$
Recursively define a sequence $\left(\tau_m\right)_{m\geq1}=\left(\tau_m(\hat{u})\right)_{m\geq1}$ of random times, which are referred to as ``regeneration times" (relative to $\hat{u}$), by 
\begin{align*}
\tau_1&:=\inf\left\{j>0:\langle X_i,\hat{u}\rangle<\langle X_j,\hat{u}\rangle\leq\langle X_k,\hat{u}\rangle\mbox{ for all }i,k\mbox{ with }i<j<k\right\}\quad\mbox{and}\\
\tau_{m}&:=\inf\left\{j>\tau_{m-1}:\langle X_i,\hat{u}\rangle<\langle X_j,\hat{u}\rangle\leq\langle X_k,\hat{u}\rangle\mbox{ for all }i,k\mbox{ with }i<j<k\right\}
\end{align*}
for every $m\geq2$.
If the walk is directionally transient relative to $\hat{u}$, i.e., if
\begin{equation}\label{transience}
P_o\left(\lim_{n\to\infty}\langle X_n,\hat{u}\rangle=\infty\right)=1,
\end{equation}
then $P_o(\beta=\infty)>0$ and $P_o\left(\tau_m<\infty\right)=1$ for every $m\geq1$.
As shown in \cite{SznitmanZerner}, the significance of $\left(\tau_m\right)_{m\geq1}$ is due to the fact that $$\left(X_{\tau_m+1}-X_{\tau_m},X_{\tau_m+2}-X_{\tau_m},\ldots,X_{\tau_{m+1}}-X_{\tau_m}\right)_{m\geq1}$$ is an i.i.d.\ sequence under $P_o$ when
\begin{equation}\label{i.i.d.}
\omega=(\omega_x)_{x\in\mathbb{Z}^d}\mbox{ is an i.i.d.\ collection}.
\end{equation}

\begin{definition}
RWRE is said to satisfy Sznitman's transience condition (\textbf{T}) relative to a unit vector $\hat{u}\in\mathcal{S}^{d-1}$ if (\ref{transience}) holds and
\begin{equation}\label{moment}
E_o\left[\sup_{1\leq i\leq\tau_1}\exp\left\{c_1\left|X_i\right|\right\}\right]<\infty\mbox{ for some }c_1>0.
\end{equation}
\end{definition}

The following theorem lists some of the important facts regarding condition (\textbf{T}).

\begin{theorem}\label{bombagibi}
Consider RWRE on $\mathbb{Z}^d$. Assume (\ref{nearestneighbor}), (\ref{ellipticity}) and (\ref{i.i.d.}). Take a unit vector $\hat{u}\in\mathcal{S}^{d-1}$.
\begin{itemize}
\item[(a)] For $d=1$, (\ref{transience}) implies (\ref{moment}). Hence, (\textbf{T}) is equivalent to (\ref{transience}). (See \cite{SznitmanT}, Proposition 2.6.)
The LLN holds with limiting velocity
\begin{equation}\label{huseysh}
\xi_o=\frac{E_o\left[\left.X_{\tau_1}\right|\beta=\infty\right]}{E_o\left[\left.\tau_1\right|\beta=\infty\right]}
\end{equation}
which can be zero.
\item[(b)] For $d\geq1$, if the walk is non-nestling relative to $\hat{u}$, then
\begin{equation}\label{isimlazimkesin}
E_o\left[\exp\left\{c_2\tau_1\right\}\right]<\infty
\end{equation} for some $c_2>0$. In particular, (\textbf{T}) is satisfied. (See \cite{SznitmanSlowdown}, Theorem 2.1.)
\item[(c)] For $d\geq2$, if (\textbf{T}) holds relative to $\hat{u}$, then all the $P_o$-moments of $\tau_1$ are finite. This implies a LLN and an averaged central limit theorem. The LLN velocity $\xi_o$ is given by the formula in (\ref{huseysh}), and it satisfies $\langle\xi_o,\hat{u}\rangle>0$. (See \cite{SznitmanT}, Theorems 3.4 and 3.6.)
\end{itemize}
\end{theorem}

\subsection{Our results}\label{israildeyim}

It follows from Theorem \ref{aLDPgeneric} and Varadhan's lemma (see \cite{DZ}) that 
\begin{equation}\label{kazimkoyuncu}
\Lambda_a(\theta):=\lim_{n\to\infty}\frac{1}{n}\log E_o\left[\exp\{\langle\theta,X_n\rangle\}\right]=\sup_{\xi\in\mathbb{R}^d}\left\{\langle\theta,\xi\rangle - I_a(\xi)\right\}
\end{equation} for every $\theta\in\mathbb{R}^d$. Hence, $\Lambda_a=I_a^*$, the convex conjugate of $I_a$.

With $c_1$ and $c_2$ as in (\ref{moment}) and (\ref{isimlazimkesin}), define
\begin{equation}\label{nebiliyim}
\mathcal{C}:=\left\{\begin{array}{ll}
\left\{\theta\in\mathbb{R}^d:|\theta|<c_2/2\right\}&\mbox{if the walk is non-nestling,}\\
\left\{\theta\in\mathbb{R}^d:|\theta|<c_1\,, \Lambda_a(\theta)>0\right\}&\mbox{if the walk is nestling and condition (\textbf{T}) holds.}
\end{array}\right.
\end{equation}
In the latter case, as we will see, $\Lambda_a$ is a nonnegative convex function, and $\mathcal{C}$ is nothing but an open ball minus a convex set.

We start Section \ref{aLDPregularitysection} by obtaining a series of intermediate results including

\begin{lemma}\label{Caveraged}
Consider RWRE on $\mathbb{Z}^d$. Assume (\ref{nearestneighbor}), (\ref{ellipticity}) and (\ref{i.i.d.}). If (\textbf{T}) holds relative to some $\hat{u}\in\mathcal{S}^{d-1}$, then $\Lambda_a$ is analytic on $\mathcal{C}$.
Moreover, the Hessian $\mathcal{H}_a$ of $\Lambda_a$ is positive definite on $\mathcal{C}$.
\end{lemma}

We then use (\ref{kazimkoyuncu}) and convex duality to establish
\begin{theorem}\label{qual}
Under the assumptions of Lemma \ref{Caveraged}, the averaged rate function $I_a$ is strictly convex and analytic on the non-empty open set
\begin{equation}\label{tabiyanig}
\mathcal{A}:=\{\nabla\Lambda_a(\theta):\theta\in\mathcal{C}\}.
\end{equation}
\begin{itemize}
\item[(a)] If the walk is non-nestling, then $\mathcal{A}$ contains $\xi_o$, the LLN velocity.
\item[(b)] If the walk is nestling and $d=1$, then $\xi_o\in\partial\mathcal{A}$.
\item[(c)] If the walk is nestling and $d\geq2$, then
\begin{itemize}
\item[(i)] there exists a $(d-1)$-dimensional smooth surface patch $\mathcal{A}^b$ such that $\xi_o\in\mathcal{A}^b\subset\partial\mathcal{A}$, and
\item[(ii)] the unit vector $\eta_o$ normal to $\mathcal{A}^b$ (and pointing inside $\mathcal{A}$) at $\xi_o$ satisfies $\langle\eta_o,\xi_o\rangle>0$. (Roughly speaking, $\mathcal{A}$ is facing away from the origin.)
\end{itemize}
\end{itemize}
\end{theorem}

\begin{remark}
After making this work available online as part of \cite{YilmazThesis}, we learned that Peterson \cite{PetersonThesis} independently proved Theorem \ref{qual} for non-nestling walks. His technique is somewhat different from ours since it involves first considering large deviations for the joint process of regeneration times and positions. Later, using that technique, Peterson and Zeitouni \cite{JonOfer} reproduced Theorem \ref{qual} in its full generality. Plus, in the nestling case, they showed that
\begin{equation}\label{lineeredebiyati}
\mbox{$I_a(t\xi)=tI_a(\xi)$ for every $\xi\in\mathcal{A}^b$ and $t\in[0,1]$.}
\end{equation}

Under the assumptions of Theorem \ref{qual}, when $d\geq4$, we recently proved in \cite{YilmazQA} that $I_a=I_q$ on a closed set whose interior contains $\{\xi\neq0:I_a(\xi)=0\}$. Also, we gave an alternative proof of (\ref{lineeredebiyati}).
\end{remark}

In Section \ref{aLDPminimizersection}, we identify the unique minimizer in (\ref{divaneasik}) for every $\xi\in\mathcal{A}$. The natural interpretation is that this minimizer gives the distribution of the RWRE path under $P_o$ when the particle is conditioned to escape to infinity with mean velocity $\xi$. 

\begin{definition}\label{definemuyuanan}
Denote the random steps of the particle by $(Z_n)_{n\geq1}:=(X_n-X_{n-1})_{n\geq1}$. Assume (\ref{nearestneighbor}), (\ref{ellipticity}), (\ref{i.i.d.}) and (\textbf{T}). The Hessian $\mathcal{H}_a$ of $\Lambda_a$ is positive definite on $\mathcal{C}$ by Lemma \ref{Caveraged}. Hence, for every $\xi\in\mathcal{A}$, there exists a unique $\theta\in\mathcal{C}$ satisfying $\xi=\nabla\Lambda_a(\theta)$. For every $K\in\mathbb{N}$, take any bounded function $f:U^{\mathbb{N}}\rightarrow\mathbb{R}$ such that $f((z_i)_{i\geq1})$ is independent of $(z_i)_{i>K}$. Define a probability measure $\bar{\mu}_\xi^\infty$ on $U^\mathbb{N}$ by setting
\begin{equation}
\int\!\! f\mathrm{d}\bar{\mu}_\xi^\infty:=\frac{E_o\left[\left.\sum_{j=0}^{\tau_1 -1}f((Z_{j+i})_{i\geq1})\ \exp\{\langle\theta,X_{\tau_{K}}\rangle - \Lambda_a(\theta)\tau_{K}\}\,\right|\,\beta=\infty\right]}{E_o\left[\left.\tau_1\ \exp\{\langle\theta,X_{\tau_1}\rangle - \Lambda_a(\theta)\tau_1\}\,\right|\,\beta=\infty\right]}.\label{muyucananan}
\end{equation}

\end{definition}

\begin{theorem}\label{sevval}
Assume (\ref{nearestneighbor}), (\ref{ellipticity}), (\ref{i.i.d.}) and (\textbf{T}). Recall (\ref{tabiyanig}) and Definition \ref{definemuyuanan}. For every $\xi\in\mathcal{A}$, $\bar{\mu}_\xi^\infty$ induces a transient process with stationary and ergodic increments via the map $$(z_1,z_2,z_3,\ldots)\mapsto(z_1,z_1+z_2,z_1+z_2+z_3,\ldots).$$ Extend this process to a probability measure on doubly infinite paths $(x_i)_{i\in\mathbb{Z}}$, and refer to its restriction to $W_{\infty}^{tr}$ as $\mu_\xi^\infty$. With this notation, $\mu_\xi^\infty$ is the unique minimizer of (\ref{divaneasik}).
\end{theorem}

\section{Strict convexity and analyticity}\label{aLDPregularitysection}

Assume (\ref{nearestneighbor}), (\ref{ellipticity}) and (\ref{i.i.d.}). If the walk is non-nestling, then (\ref{nonmumu}) is satisfied for some $\hat{u}\in\mathcal{S}^{d-1}$. If the walk is nestling, assume that (\textbf{T}) holds relative to some $\hat{u}\in\mathcal{S}^{d-1}$.

\subsection{Logarithmic moment generating function}

Recall (\ref{kazimkoyuncu}). By Jensen's inequality, 
$$\langle\theta,\xi_o\rangle=\lim_{n\to\infty}\frac{1}{n} E_o\left[\langle\theta,X_n\rangle\right]\leq\lim_{n\to\infty}\frac{1}{n}\log E_o\left[\exp\{\langle\theta,X_n\rangle\}\right]=\Lambda_a(\theta)\leq\lim_{n\to\infty}\frac{1}{n}\log E_o\left[\mathrm{e}^{|\theta|n}\right]=|\theta|.$$

\begin{lemma}\label{phillysh}
$E_o\left[\left.\exp\{\langle\theta,X_{\tau_1}\rangle-\Lambda_a(\theta)\tau_1\}\right|\beta=\infty\right]\leq1$ for every $\theta\in\mathbb{R}^d$.
\end{lemma}
\begin{proof}

For every $n\geq1$, $\theta\in\mathbb{R}^d$ and $\epsilon>0$,
\begin{align*}
E_o\left[\exp\left\{\langle\theta,X_{\tau_n}\rangle-(\Lambda_a(\theta)+\epsilon)\tau_n\right\}\right]&=\sum_{i=n}^\infty E_o\left[\exp\left\{\langle\theta,X_{\tau_n}\rangle-(\Lambda_a(\theta)+\epsilon)\tau_n\right\},\tau_n=i\right]\\
&\leq\sum_{i=n}^\infty E_o\left[\exp\left\{\langle\theta,X_i\rangle-(\Lambda_a(\theta)+\epsilon)i\right\}\right]\\
&=\sum_{i=n}^\infty\mathrm{e}^{o(i)-\epsilon i}\leq\sum_{i=n}^\infty\mathrm{e}^{-\epsilon i/2}=\mathrm{e}^{-\epsilon n/2}\left(1-\mathrm{e}^{-\epsilon/2}\right)^{-1}
\end{align*} when $n$ is sufficiently large. On the other hand, 
\begin{align*}
&E_o\left[\exp\left\{\langle\theta,X_{\tau_n}\rangle-(\Lambda_a(\theta)+\epsilon)\tau_n\right\}\right]\\
&\quad=E_o\left[\exp\left\{\langle\theta,X_{\tau_1}\rangle-(\Lambda_a(\theta)+\epsilon)\tau_1\right\}\right]E_o\left[\left.\exp\left\{\langle\theta,X_{\tau_1}\rangle-(\Lambda_a(\theta)+\epsilon)\tau_1\right\}\right|\beta=\infty\right]^{n-1}
\end{align*} by the renewal structure. Hence, $E_o\left[\left.\exp\left\{\langle\theta,X_{\tau_1}\rangle-(\Lambda_a(\theta)+\epsilon)\tau_1\right\}\right|\beta=\infty\right]\leq\mathrm{e}^{-\epsilon/2}$. The desired result is obtained by taking $\epsilon\to 0$ and applying the monotone convergence theorem.
\end{proof}

Recall (\ref{nebiliyim}). For every $\epsilon>0$, it is clear that $E_o\left[\left.\exp\{\langle\epsilon\hat{u},X_{\tau_1}\rangle\}\right|\beta=\infty\right]>1$. This, in combination with Lemma \ref{phillysh}, implies that $\Lambda_a(\epsilon\hat{u})>0$. Therefore, $\mathcal{C}$ is non-empty.

In the nestling case, $I_a(0)=0$, cf.\ \cite{Raghu}. It follows from (\ref{kazimkoyuncu}) and convex duality that 
\begin{equation}\label{bbqmangal}
0=I_a(0)=\sup_{\theta\in\mathbb{R}^d}\left\{\langle\theta,0\rangle - \Lambda_a(\theta)\right\}=-\inf_{\theta\in\mathbb{R}^d}\Lambda_a(\theta).
\end{equation}
In other words, $\Lambda_a(\theta)\geq0$ for every $\theta\in\mathbb{R}^d$. The zero-level set $\left\{\theta\in\mathbb{R}^d:\,\Lambda_a(\theta)=0\right\}$ of the convex function $\Lambda_a$ is convex, and $\mathcal{C}$ is an open ball minus this convex set.

\begin{lemma}\label{phillysh_one}
$E_o\left[\left.\exp\{\langle\theta,X_{\tau_1}\rangle-\Lambda_a(\theta)\tau_1\}\right|\beta=\infty\right]=1$ for every $\theta\in\mathcal{C}$.
\end{lemma}

\begin{proof}
Adopt the convention that $\tau_o=0$. For every $n\geq1$, $\theta\in\mathcal{C}$  and $r\in\mathbb{R}$,
\begin{align*}
&E_o\left[\exp\{\langle\theta,X_n\rangle-rn\}\right]=\sum_{m=0}^n\sum_{i=0}^n E_o\left[\exp\{\langle\theta,X_n\rangle-rn\},\tau_m\leq n<\tau_{m+1},n-\tau_m=i\right]\\
&\quad=\sum_{m=0}^n\sum_{i=0}^n E_o\left[\exp\{\langle\theta,X_{\tau_m}\rangle-r\tau_m\},\tau_m=n-i\right]E_o\left[\left.\exp\{\langle\theta,X_i\rangle-ri\},i<\tau_1\right|\beta=\infty\right]\\
&\quad\leq\sum_{m=0}^\infty E_o\left[\exp\{\langle\theta,X_{\tau_m}\rangle-r\tau_m\}\right]E_o\left[\left.\sup_{0\leq i<\tau_1}\exp\{\langle\theta,X_i\rangle-ri\}\right|\beta=\infty\right]\\
&\quad= E_o\left[\left.\sup_{0\leq i<\tau_1}\exp\{\langle\theta,X_i\rangle-ri\}\right|\beta=\infty\right]\\
&\quad\quad\times\left(1+E_o\left[\exp\{\langle\theta,X_{\tau_1}\rangle-r\tau_1\}\right]\sum_{m=0}^\infty E_o\left[\left.\exp\{\langle\theta,X_{\tau_1}\rangle-r\tau_1\}\right|\beta=\infty\right]^m\right)\\
&\quad<\infty
\end{align*}
whenever
\begin{align}
&E_o\left[\sup_{0\leq i<\tau_1}\exp\{\langle\theta,X_i\rangle-ri\}\right]<\infty,\quad E_o\left[\exp\{\langle\theta,X_{\tau_1}\rangle-r\tau_1\}\right]<\infty,\quad\mbox{and}\label{needthese}\\
&E_o\left[\left.\exp\{\langle\theta,X_{\tau_1}\rangle-r\tau_1\}\right|\beta=\infty\right]<1.\label{cannothold}
\end{align}
Therefore, (\ref{needthese}) and (\ref{cannothold}) imply that 
\begin{equation}\label{nottrue}
\Lambda_a(\theta)-r=\lim_{n\to\infty}\frac{1}{n}\log E_o\left[\exp\{\langle\theta,X_n\rangle-rn\}\right]\leq0.
\end{equation}

If the walk is non-nestling, then there exists an $\epsilon>0$ such that $|\theta|+|\Lambda_a(\theta)|+\epsilon\leq2|\theta|+\epsilon<c_2$. Take $r=\Lambda_a(\theta)-\epsilon$. Then, (\ref{needthese}) follows from Theorem \ref{bombagibi}. Since (\ref{nottrue}) is false, (\ref{cannothold}) is false as well. In other words,
\begin{equation}\label{dogrusubu}
1\leq E_o\left[\left.\exp\{\langle\theta,X_{\tau_1}\rangle-(\Lambda_a(\theta)-\epsilon)\tau_1\}\right|\beta=\infty\right]<\infty.
\end{equation}

If the walk is nestling, then $\Lambda_a(\theta)>0$ and there exists an $\epsilon>0$ such that $\Lambda_a(\theta)-\epsilon>0$. Take $r=\Lambda_a(\theta)-\epsilon$. Then, (\ref{needthese}) follows from (\ref{moment}). Since (\ref{nottrue}) is false, (\ref{dogrusubu}) is true.

Clearly, (\ref{dogrusubu}) and the monotone convergence theorem imply that $E_o\left[\left.\exp\{\langle\theta,X_{\tau_1}\rangle-\Lambda_a(\theta)\tau_1\}\right|\beta=\infty\right]\geq1$. Combined with Lemma \ref{phillysh}, this gives the desired result.
\end{proof}

\begin{lemma}\label{geldiiksh}

Assume that the walk is nestling. With $c_1$ as in (\ref{moment}), define
\begin{equation}\label{nebiliyimiste}
\mathcal{C}^b:= \left\{\theta\in\partial\mathcal{C}:|\theta|<c_1\right\}.
\end{equation}
\begin{itemize}
\item[(a)] If $|\theta|<c_1$, then $\theta\not\in\mathcal{C}$ if and only if $E_o\left[\left.\exp\{\langle\theta,X_{\tau_1}\rangle\}\right|\beta=\infty\right]\leq1$.
\item[(b)] If $|\theta|<c_1$, then $\theta\in\mathcal{C}^b$ if and only if $E_o\left[\left.\exp\{\langle\theta,X_{\tau_1}\rangle\}\right|\beta=\infty\right]=1$. 
\end{itemize}
\end{lemma}

\begin{proof}
Recall that $\Lambda_a(\theta)\geq0$ for every $\theta\in\mathbb{R}^d$ by (\ref{bbqmangal}).
If $|\theta|<c_1$ and $\theta\not\in\mathcal{C}$, then $\Lambda_a(\theta)=0$ and $E_o\left[\left.\exp\{\langle\theta,X_{\tau_1}\rangle\}\right|\beta=\infty\right]\leq1$ by Lemma \ref{phillysh}. Conversely, if $|\theta|<c_1$ and $E_o\left[\left.\exp\{\langle\theta,X_{\tau_1}\rangle\}\right|\beta=\infty\right]\leq1$, then $\Lambda_a(\theta)>0$ cannot be true because it would imply that
$$1=E_o\left[\left.\exp\{\langle\theta,X_{\tau_1}\rangle-\Lambda_a(\theta)\tau_1\}\right|\beta=\infty\right]< E_o\left[\left.\exp\{\langle\theta,X_{\tau_1}\rangle\}\right|\beta=\infty\right]\leq1$$ by Lemma \ref{phillysh_one}. Hence, $\Lambda_a(\theta)=0$. This proves part (a).

If $\theta\in\mathcal{C}^b$, then $\Lambda_a(\theta)=0$. Take $\theta_n\in\mathcal{C}$ such that $\theta_n\to\theta$. It follows from Lemma \ref{phillysh_one} that $$E_o\left[\left.\exp\{\langle\theta_n,X_{\tau_1}\rangle-\Lambda_a(\theta_n)\tau_1\}\right|\beta=\infty\right]=1.$$ Since $\Lambda_a$ is continuous at $\theta$, $E_o\left[\left.\exp\{\langle\theta,X_{\tau_1}\rangle\}\right|\beta=\infty\right]=1$ by (\ref{moment}) and the dominated convergence theorem.

$\Lambda_a$ is a convex function and $\{\theta\in\mathbb{R}^d:\Lambda_a(\theta)=0\}$ is convex. If $\theta$ is an interior point of this set, then $\theta=t\theta_1+(1-t)\theta_2$ for some $t\in(0,1)$ and $\theta_1,\theta_2\in\mathbb{R}^d$ such that $\theta_1\neq\theta_2$ and $E_o\left[\left.\exp\{\langle\theta_i,X_{\tau_1}\rangle\}\right|\beta=\infty\right]\leq1$ for $i=1,2$. By Jensen's inequality, $E_o\left[\left.\exp\{\langle\theta,X_{\tau_1}\rangle\}\right|\beta=\infty\right]<1$. The contraposition of this argument concludes the proof of part (b).
\end{proof}

\begin{proof}[Proof of Lemma \ref{Caveraged}]
Consider the function $\psi:\mathbb{R}^d\times\mathbb{R}\to\mathbb{R}$ defined as 
\begin{equation}\label{psish}
\psi(\theta,r):=E_o\left[\left.\exp\{\langle\theta,X_{\tau_1}\rangle-r\tau_1\}\right|\beta=\infty\right].
\end{equation} When $\theta\in\mathcal{C}$ and $|r-\Lambda_a(\theta)|$ is small enough, it follows from (\ref{moment}), Theorem \ref{bombagibi} and Lemma \ref{phillysh_one} that $\psi(\theta,r)<\infty$ and $\psi(\theta,\Lambda_a(\theta))=1$. Clearly, $(\theta,r)\mapsto\psi(\theta,r)$ is analytic at such $(\theta,r)$.

If the walk is non-nestling or if it is nestling but $d\geq2$, then all the $P_o$-moments of $\tau_1$ are finite and $E_o\left[\left.\tau_1\exp\{\langle\theta,X_{\tau_1}\rangle-\Lambda_a(\theta)\tau_1\}\right|\beta=\infty\right]<\infty$ by H\"older's inequality and Theorem \ref{bombagibi}.

If the walk is nestling and $d\geq1$, then $\Lambda_a(\theta)>0$, and (\ref{moment}) implies that
\begin{align*}
E_o\left[\left.\tau_1\exp\{\langle\theta,X_{\tau_1}\rangle-\Lambda_a(\theta)\tau_1\}\right|\beta=\infty\right]&\leq\left(\sup_{t\geq0}t\mathrm{e}^{-\Lambda_a(\theta)t}\right)E_o\left[\left.\exp\{\langle\theta,X_{\tau_1}\rangle\}\right|\beta=\infty\right]\\
&=(\mathrm{e}\Lambda_a(\theta))^{-1}E_o\left[\left.\exp\{\langle\theta,X_{\tau_1}\rangle\}\right|\beta=\infty\right]<\infty.
\end{align*}

In both cases, Lemma \ref{phillysh_one} implies that $$E_o\left[\left.\tau_1\exp\{\langle\theta,X_{\tau_1}\rangle-\Lambda_a(\theta)\tau_1\}\right|\beta=\infty\right]\geq E_o\left[\left.\exp\{\langle\theta,X_{\tau_1}\rangle-\Lambda_a(\theta)\tau_1\}\right|\beta=\infty\right]=1.$$ 
Therefore, $$\left.\partial_r\psi(\theta,r)\right|_{r=\Lambda_a(\theta)}=-E_o\left[\left.\tau_1\exp\{\langle\theta,X_{\tau_1}\rangle-\Lambda_a(\theta)\tau_1\}\right|\beta=\infty\right]\in\left(-\infty,-1\right],$$ and $\Lambda_a$ is analytic on $\mathcal{C}$ by the analytic implicit function theorem. (See \cite{implicit}, Theorem 6.1.2.)

Differentiating both sides of $\psi(\theta,\Lambda_a(\theta))=1$ with respect to $\theta$ gives
\begin{equation}\label{jacobiansh}
E_o\left[\left.\left(X_{\tau_1}-\nabla\Lambda_a(\theta)\tau_1\right)\exp\{\langle\theta,X_{\tau_1}\rangle-\Lambda_a(\theta)\tau_1\}\right|\beta=\infty\right]=0
\end{equation} and
\begin{equation}\label{gradientsh}
\nabla\Lambda_a(\theta)=\frac{E_o\left[\left.X_{\tau_1}\exp\{\langle\theta,X_{\tau_1}\rangle-\Lambda_a(\theta)\tau_1\}\right|\beta=\infty\right]}{E_o\left[\left.\tau_1\exp\{\langle\theta,X_{\tau_1}\rangle-\Lambda_a(\theta)\tau_1\}\right|\beta=\infty\right]}.
\end{equation}
Differentiating both sides of (\ref{jacobiansh}), we see that the Hessian $\mathcal{H}_a$ of $\Lambda_a$ satisfies
\begin{equation}\label{hessiansh}
\langle v_1,\mathcal{H}_a(\theta)v_2\rangle=\frac{E_o\left[\left.\langle X_{\tau_1}-\nabla\Lambda_a(\theta)\tau_1,v_1\rangle\langle X_{\tau_1}-\nabla\Lambda_a(\theta)\tau_1,v_2\rangle\,\exp\{\langle\theta,X_{\tau_1}\rangle-\Lambda_a(\theta)\tau_1\}\right|\beta=\infty\right]}{E_o\left[\left.\tau_1\exp\{\langle\theta,X_{\tau_1}\rangle-\Lambda_a(\theta)\tau_1\}\right|\beta=\infty\right]}
\end{equation} for any two vectors $v_1\in\mathbb{R}^d$ and $v_2\in\mathbb{R}^d$.

We already saw that the denominator of the RHS of (\ref{hessiansh}) is finite. A similar argument shows that the numerator is finite as well. (\ref{ellipticity}) ensures that the numerator is positive when $v_1=v_2$. Thus, $\mathcal{H}_a$ is positive definite on $\mathcal{C}$.
\end{proof}

\subsection{Rate function} 

\begin{proof}[Proof of Theorem \ref{qual}]
$\Lambda_a$ is analytic on $\mathcal{C}$, and the Hessian $\mathcal{H}_a$ of $\Lambda_a$ is positive definite on $\mathcal{C}$, cf.\ Lemma \ref{Caveraged}. Therefore, for every $\xi\in\mathcal{A}$, there exists a unique $\theta=\theta(\xi)\in\mathcal{C}$ such that $\xi=\nabla\Lambda_a(\theta)$. $\mathcal{A}$ is open since it is the pre-image of $\mathcal{C}$ under the map $\xi\mapsto\theta(\xi)$ which is analytic by the inverse function theorem. Since \begin{equation}\label{herkoyun}
I_a(\xi)=\sup_{\theta'\in\mathbb{R}^d}\left\{\langle\theta',\xi\rangle-\Lambda_a(\theta')\right\}=\langle\theta(\xi),\xi\rangle-\Lambda_a(\theta(\xi)),
\end{equation} we conclude that $I_a$ is analytic at $\xi$. Differentiating (\ref{herkoyun}) twice with respect to $\xi$ shows that the Hessian of $I_a$ at $\xi$ is equal to $\mathcal{H}_a(\theta(\xi))^{-1}$, a positive definite matrix. Therefore, $I_a$ is strictly convex on $\mathcal{A}$.

If the walk is non-nestling, then $0\in\mathcal{C}$ and $$\xi_o=\frac{E_o\left[\left.X_{\tau_1}\right|\beta=\infty\right]}{E_o\left[\left.\tau_1\right|\beta=\infty\right]}=\nabla\Lambda_a(0)\in\mathcal{A}$$ by (\ref{huseysh}) and (\ref{gradientsh}). This proves part (a).

The rest of this proof focuses on the nestling case. When $d=1$, Lemma \ref{geldiiksh} implies that $0\in\partial\mathcal{C}$. Take any $(\theta_n)_{n\geq1}$ with $\theta_n\in\mathcal{C}$ such that $\theta_n\to 0$. Then, any limit point of $(\nabla\Lambda_a(\theta_n))_{n\geq1}$ belongs to $\partial\mathcal{A}$. (\ref{huseysh}) and (\ref{gradientsh}) imply that
\begin{align}
\limsup_{n\to\infty}\nabla\Lambda_a(\theta_n)&=\limsup_{n\to\infty}\frac{E_o\left[\left.X_{\tau_1}\exp\{\langle\theta_n,X_{\tau_1}\rangle-\Lambda_a(\theta_n)\tau_1\}\right|\beta=\infty\right]}{E_o\left[\left.\tau_1\exp\{\langle\theta_n,X_{\tau_1}\rangle-\Lambda_a(\theta_n)\tau_1\}\right|\beta=\infty\right]}\label{nediyonsh}\\&\leq\frac{E_o\left[\left.X_{\tau_1}\right|\beta=\infty\right]}{E_o\left[\left.\tau_1\right|\beta=\infty\right]}=\xi_o,\label{nediyonggsh}
\end{align}
where we assume WLOG that $\hat{u}=1$. The numerator in (\ref{nediyonsh}) converges to the numerator in (\ref{nediyonggsh}) by (\ref{moment}) and the dominated convergence theorem. The denominator in (\ref{nediyonggsh}) bounds the liminf of the denominator in (\ref{nediyonsh}) by Fatou's lemma. $[0,\xi_o]\cap\mathcal{A}$ is empty since $I_a$ is linear on $[0,\xi_o]$. (This only makes sense if $\xi_o>0$. However, when $\xi_o=0$, it is clear from (\ref{gradientsh}) that $0\notin\mathcal{A}$.) Therefore, $\liminf_{n\to\infty}\nabla\Lambda_a(\theta_n)\geq\xi_o$. Hence, $\xi_o=\lim_{n\to\infty}\nabla\Lambda_a(\theta_n)\in\partial\mathcal{A}$.

When $d\geq2$, (\ref{gradientsh}), H\"older's inequality and Theorem \ref{bombagibi} imply that $\nabla\Lambda_a$ extends smoothly to $\mathcal{C}\cup\mathcal{C}^b$. Refer to the extension by $\overline{\nabla\Lambda_a}$. Define $\mathcal{A}^b:=\left\{\overline{\nabla\Lambda_a}(\theta): \theta\in\mathcal{C}^b\right\}$.
Note that $0\in\mathcal{C}^b\subset\partial\mathcal{C}$ by Lemma \ref{geldiiksh}, and $\xi_o=\overline{\nabla\Lambda_a}(0)\in\mathcal{A}^b\subset\partial\mathcal{A}$.

The map $\theta\mapsto\psi(\theta,0)=E_o\left[\left.\exp\{\langle\theta,X_{\tau_1}\rangle\}\right|\beta=\infty\right]$ is analytic on $\{\theta\in\mathbb{R}^d: |\theta|<c_1\}$. For every $\theta\in\mathcal{C}^b$,
$$\langle\nabla_\theta\psi(\theta,0),\hat{u}\rangle=E_o\left[\left.\langle X_{\tau_1},\hat{u}\rangle\exp\{\langle\theta,X_{\tau_1}\rangle\}\right|\beta=\infty\right]>0.$$
Lemma \ref{geldiiksh} and the implicit function theorem imply that $\mathcal{C}^b$ is the graph of an analytic function. Therefore, $\mathcal{A}^b$ is a $(d-1)$-dimensional smooth surface patch. Note that $$\left.\nabla_\theta\psi(\theta,0)\right|_{\theta=0}=E_o\left[\left.X_{\tau_1}\right|\beta=\infty\right]=E_o\left[\left.\tau_1\right|\beta=\infty\right]\xi_o$$ is normal to $\mathcal{C}^b$ at $0$. Refer to the extension of $\mathcal{H}_a$ to $\mathcal{C}\cup\mathcal{C}^b$ as $\overline{\mathcal{H}_a}$. The unit vector $\eta_o$ normal to $\mathcal{A}^b$ (and pointing inside $\mathcal{A}$) at $\xi_o$ is $c\overline{\mathcal{H}_a}(0)^{-1}\xi_o$ for some $c>0$ by the chain rule. It is clear from (\ref{ellipticity}) and (\ref{hessiansh}) that  $$\langle\eta_o,\xi_o\rangle=c\langle\xi_o,\overline{\mathcal{H}_a}(0)^{-1}\xi_o\rangle>0.\qedhere$$
\end{proof}

\section{Minimizer of Varadhan's variational formula}\label{aLDPminimizersection}

\subsection{Existence of the minimizer}

Varadhan's variational formula for the rate function $I_a$ at any $\xi\neq0$ is
\begin{equation}\label{divaneasiksh}
I_a(\xi)=\inf_{\substack{\mu\in\mathcal{E}:\\m(\mu)=\xi}}\mathfrak{I}_a(\mu).
\end{equation} 
Recall (\ref{parisolur}). There exists a measurable function $\hat{q}:W_\infty^{\mathrm{tr}}\times U\to[0,1]$ such that $\hat{q}(\cdot,z)=q_\mu(\cdot,z)$ holds $\mu$-a.s.\ for every $\mu\in\mathcal{I}$ and $z\in U$. (See \cite{DV4}, Lemma 3.4.) The formula (\ref{hayirsh}) for $\mathfrak{I}_a$ can be written as
\begin{equation}\label{evetsh}
\mathfrak{I}_a(\mu)=\int_{W_\infty^{\mathrm{tr}}}\left[\sum_{z\in U}\hat{q}(w,z)\log\frac{\hat{q}(w,z)}{q(w,z)}\right]\,\mathrm{d}\mu(w).
\end{equation}
Therefore, $\mathfrak{I}_a$ is affine linear on $\mathcal{I}$.

\begin{lemma}
$$I_a(\xi)=\inf_{\substack{\mu\in\mathcal{I}:\\m(\mu)=\xi}}\mathfrak{I}_a(\mu).$$
\end{lemma}
\begin{proof}
By the definition of $I_a$ in (\ref{divaneasiksh}), $$I_a(\xi)\geq\inf_{\substack{\mu\in\mathcal{I}:\\m(\mu)=\xi}}\mathfrak{I}_a(\mu)$$ is clear. To establish the reverse inequality, take any $\mu\in\mathcal{I}$ with $m(\mu)=\xi$. Since $\mathcal{E}$ is the set of extremal points of $\mathcal{I}$, $\mu$ can be expressed as $$\mu=\int_{\mathcal{E}_o}\alpha\,\mathrm{d}\hat{\mu}(\alpha) + \int_{\mathcal{E}\backslash\mathcal{E}_o}\alpha\,\mathrm{d}\hat{\mu}(\alpha)=\int_{\mathcal{E}_o}\alpha\,\mathrm{d}\hat{\mu}(\alpha) + (1-\hat{\mu}(\mathcal{E}_o))\tilde{\mu}$$ where $\mathcal{E}_o:=\{\alpha\in\mathcal{E}:m(\alpha)\neq0\}$, $\hat{\mu}$ is some probability measure on $\mathcal{E}$, and $\tilde{\mu}\in\mathcal{I}$ with $m(\tilde{\mu})=0$. Then,
\begin{align}
\mathfrak{I}_a(\mu)&=\int_{\mathcal{E}_o}\mathfrak{I}_a(\alpha)\,\mathrm{d}\hat{\mu}(\alpha) + (1-\hat{\mu}(\mathcal{E}_o))\mathfrak{I}_a(\tilde{\mu})\label{miyash}\\
&\geq\int_{\mathcal{E}_o}I_a(m(\alpha))\,\mathrm{d}\hat{\mu}(\alpha) + (1-\hat{\mu}(\mathcal{E}_o))I_a(0)\label{miyaash}\\
&\geq I_a(\xi).\label{miyaaash}
\end{align} The equality in (\ref{miyash}) uses the affine linearity of $\mathfrak{I}_a$. (\ref{miyaash}) follows from two facts: (i) $\mathfrak{I}_a(\alpha)\geq I_a(m(\alpha))$, and (ii) $\mathfrak{I}_a(\tilde{\mu})\geq I_a(0)$. The first fact is immediate from the definition of $I_a$. See Lemma 7.2 of \cite{Raghu} for the proof of the second fact. Finally, the convexity of $I_a$ gives (\ref{miyaaash}).
\end{proof}

\begin{lemma}\label{tugish}
If $I_a$ is strictly convex at $\xi$, then the infimum in (\ref{divaneasiksh}) is attained.
\end{lemma}

\begin{proof}
Let $W_n := \{\left(x_i\right)_{-n\leq i\leq0}:x_{i+1}-x_i\in U,\,x_o=0\}$. The simplest compactification of $W:=\cup_{n}W_n$ is $W_\infty:=\{\left(x_i\right)_{i\leq0}:x_{i+1}-x_i\in U,\,x_o=0\}$ with the product topology. However, the functions $q(\cdot,z)$ (recall (\ref{ozgurevren})) are only defined on $W_{\infty}^{\mathrm{tr}}$, and even when restricted to it they are not continuous since two walks that are identical in the immediate past are close to each other in this topology even if one of them visits $0$ in the remote past and the other one doesn't.

Section 5 of \cite{Raghu} introduces a more convenient compactification $\overline{W}$ of $W$. The functions $q(\cdot,z)$ continuously extend from $W$ to $\overline{W}$. Denote the $T^*$-invariant probability measures on $\overline{W}$ by $\overline{\mathcal{I}}$, and the extremals of $\overline{\mathcal{I}}$ by $\overline{\mathcal{E}}$. Recall that $\mathcal{E}_o:=\{\alpha\in\mathcal{E}:m(\alpha)\neq0\}$. Then, $\mathcal{E}_o\subset\mathcal{E}\subset\overline{\mathcal{E}}$ and $\mathcal{I}\subset\overline{\mathcal{I}}$. Note that the domain of the formula for $\mathfrak{I}_a$ given in (\ref{evetsh}) extends to $\overline{\mathcal{I}}$.

Take $\mu_n\in\mathcal{E}$ such that $m(\mu_n)=\xi$ and $\mathfrak{I}_a(\mu_n)\to I_a(\xi)$ as $n\to\infty$. Let $\overline{\mu}\in\overline{\mathcal{I}}$ be a weak limit point of $\mu_n$. Corollary 6.2 of \cite{Raghu} shows that $\overline{\mu}$ has a representation $$\overline{\mu}=\int_{\mathcal{E}_o}\alpha\,\mathrm{d}\hat{\mu}_1(\alpha) + (1-\hat{\mu}_1(\mathcal{E}_o))\overline{\mu}_2$$ where $\hat{\mu}_1$ is some probability measure on $\mathcal{E}_o$, and $\overline{\mu}_2\in\overline{\mathcal{I}}$ with $m(\overline{\mu}_2)=0$. Then,
\begin{align}
I_a(\xi)=\lim_{n\to\infty}\mathfrak{I}_a(\mu_n)&\geq\mathfrak{I}_a(\overline{\mu})\label{viysh}\\
&=\int_{\mathcal{E}_o}\mathfrak{I}_a(\alpha)\,\mathrm{d}\hat{\mu}_1(\alpha) + (1-\hat{\mu}_1(\mathcal{E}_o))\mathfrak{I}_a(\overline{\mu}_2)\label{viyash}\\
&\geq\int_{\mathcal{E}_o}I_a(m(\alpha))\,\mathrm{d}\hat{\mu}_1(\alpha) + (1-\hat{\mu}_1(\mathcal{E}_o))I_a(0)\label{viyaash}\\
&\geq I_a(\xi).\label{viyaaash}
\end{align} The inequality in (\ref{viysh}) follows from the lower semicontinuity of $\mathfrak{I}_a$, and the equality in (\ref{viyash}) is a consequence of the affine linearity of $\mathfrak{I}_a$. (\ref{viyaash}) relies on the fact that $\mathfrak{I}_a(\overline{\mu}_2)\geq I_a(0)$. (See Lemma 7.2 of \cite{Raghu} for the proof.) Finally, the convexity of $I_a$ gives (\ref{viyaaash}). Since $I_a$ is assumed to be strictly convex at $\xi$, $\hat{\mu}_1\left(\alpha\in\mathcal{E}_o: m(\alpha)=\xi,\, \mathfrak{I}_a(\alpha)=I_a(\xi)\right)=1$. Hence, we are done.
\end{proof}

\subsection{Formula for the unique minimizer}

Fix any $\xi\in\mathcal{A}$. Recall Definition \ref{definemuyuanan} and Theorem \ref{sevval}.

\begin{proposition}\label{bursayok}
$\bar{\mu}_\xi^\infty$ is well defined.
\end{proposition}
\begin{proof}
For every $K\in\mathbb{N}$, take any bounded function $f:U^{\mathbb{N}}\rightarrow\mathbb{R}$ such that $f((z_i)_{i\geq1})$ is independent of $(z_i)_{i>K}$. Then, $f((z_i)_{i\geq1})$ is independent of $(z_i)_{i>K'}$ for every $K'>K$ as well. So, we need to show that (\ref{muyucananan}) does not change if we replace $K$ by $K+1$. But, this is clear because
\begin{align*}
&E_o\left[\left.\sum_{j=0}^{\tau_1 -1}f((Z_{j+i})_{i\geq1})\ \exp\{\langle\theta,X_{\tau_{K+1}}\rangle - \Lambda_a(\theta)\tau_{K+1}\}\,\right|\,\beta=\infty\right]\\
&\qquad=E_o\left[\left.\sum_{j=0}^{\tau_1 -1}f((Z_{j+i})_{i\geq1})\ \exp\{\langle\theta,X_{\tau_{K}}\rangle - \Lambda_a(\theta)\tau_{K}\}\left\{\mathrm{e}^{\langle\theta,X_{\tau_{K+1}}-X_{\tau_K}\rangle - \Lambda_a(\theta)(\tau_{K+1}-\tau_K)}\right\}\,\right|\,\beta=\infty\right]\\
&\qquad=E_o\left[\left.\sum_{j=0}^{\tau_1 -1}f((Z_{j+i})_{i\geq1})\ \exp\{\langle\theta,X_{\tau_{K}}\rangle - \Lambda_a(\theta)\tau_{K}\}\,\right|\,\beta=\infty\right].
\end{align*}Explanation: In the second line of the display above, the term in $\{\cdot\}$ is independent of the others. The expectation therefore splits, and Lemma \ref{phillysh_one} implies that
\begin{align*}
&E_o\left[\left.\exp\{\langle\theta,X_{\tau_{K+1}}-X_{\tau_K}\rangle - \Lambda_a(\theta)(\tau_{K+1}-\tau_K)\}\,\right|\,\beta=\infty\right]\\
&\quad=E_o\left[\left.\exp\{\langle\theta,X_{\tau_1}\rangle - \Lambda_a(\theta)\tau_1\}\,\right|\,\beta=\infty\right]=1.\qedhere
\end{align*}
\end{proof}

The following theorem states that the empirical process $$\bar{\nu}_{n,X}^\infty := \frac{1}{n}\sum_{j=0}^{n-1}\one_{\left(Z_{j+i}\right)_{i\geq1}}$$ of the walk under $P_o$ converges to $\bar{\mu}_\xi^\infty$ when the particle is conditioned to have mean velocity $\xi$. Here, $Z_i = X_i-X_{i-1}$.

\begin{theorem}\label{averagedconditioningsh}
For every $K\in\mathbb{N}$, $f:U^{\mathbb{N}}\rightarrow\mathbb{R}$ such that $f((z_i)_{i\geq1})$ is independent of $(z_i)_{i>K}$ and bounded, and $\epsilon>0$,
\[\limsup_{\delta\to0}\limsup_{n\rightarrow\infty}\frac{1}{n}\log P_o\left(\ \left|\int\!\! f\mathrm{d}\bar{\nu}_{n,X}^\infty-\int\!\! f\mathrm{d}\bar{\mu}_\xi^\infty\right|>\epsilon\ \left|\ |\frac{X_n}{n}-\xi|\leq\delta\right.\right)<0.\]
\end{theorem}

\begin{proof}[Proof in the non-nestling case]
Since $\xi\in\mathcal{A}$, there exists a unique $\theta\in\mathcal{C}$ such that $\xi=\nabla\Lambda_a(\theta)$. Let $g(\cdot):=f(\cdot)-\int\!\! f\mathrm{d}\bar{\mu}_\xi^\infty$. Assume WLOG that $|g|\leq 1$. Then, $\int\!\! f\mathrm{d}\bar{\nu}_{n,X}^\infty-\int\!\! f\mathrm{d}\bar{\mu}_\xi^\infty = \int\!\! g\,\mathrm{d}\bar{\nu}_{n,X}^\infty =: \langle g,\bar{\nu}_{n,X}^\infty\rangle$. For any $s\in\mathbb{R}$,
\begin{align}
&E_o\left[\exp\{\langle\theta,X_n\rangle-\Lambda_a(\theta)n+ns\langle g,\bar{\nu}_{n,X}^\infty\rangle\}\right]\nonumber\\
&\qquad=E_o\left[n<\tau_{K+1},\,\exp\{\langle\theta,X_n\rangle-\Lambda_a(\theta)n+ns\langle g,\bar{\nu}_{n,X}^\infty\rangle\}\right]\label{cokazkaldish}\\&\qquad\quad+\sum_{m=K+1}^nE_o\left[\tau_m\leq n<\tau_{m+1},\,\exp\{\langle\theta,X_n\rangle-\Lambda_a(\theta)n+ns\langle g,\bar{\nu}_{n,X}^\infty\rangle\}\right]\nonumber.
\end{align}
If $|s|$ is small enough so that $2|\theta|+|s|<c_2$, then the first term in (\ref{cokazkaldish}) is bounded from above by $E_o\!\left[n<\tau_{K+1},\,\exp\{(2|\theta|+|s|)\tau_{K+1}\}\right]$ which goes to $0$ as $n\to\infty$ by Theorem \ref{bombagibi} and the monotone convergence theorem. For $j\geq0$, define
\begin{equation}\label{sabrialtintas}
G_j:=\sum_{k=\tau_j}^{\tau_{j+1}-1}g((Z_{k+i})_{i\geq1})
\end{equation} with the convention that $\tau_o=0$. Note that $G_j$ is a function of $Z_{\tau_j+1},\ldots,Z_{\tau_{j+1}+K-1}$. Therefore, $G_j$ and $G_{j+K}$ depend on disjoint sets of steps since $\tau_{j+1}+K-1\leq\tau_{j+K}$. For any $p,q\in\mathbb{R}$ with $1<p<c_2/2|\theta|$ and $1/p+1/q=1$, each term of the sum in (\ref{cokazkaldish}) can be bounded using H\"older's inequality:
\begin{align}
&E_o\left[\tau_m\leq n<\tau_{m+1},\,\exp\{\langle\theta,X_n\rangle-\Lambda_a(\theta)n+ns\langle g,\bar{\nu}_{n,X}^\infty\rangle\}\right]\nonumber\\
&\quad\leq E_o\left[\mathrm{e}^{\langle\theta,X_{\tau_m}-X_{\tau_1}\rangle-\Lambda_a(\theta)(\tau_m-\tau_1)+s\left(G_1+\cdots+G_{m-K-1}\right)+\left(2|\theta|+|s|\right)\left(\tau_1+\tau_{m+1}-\tau_m\right)+|s|\left(\tau_m-\tau_{m-K}\right)}\right]\nonumber\\
&\quad\leq E_o\left[\mathrm{e}^{\left(2|\theta|+|s|\right)\tau_1}\right]E_o\left[\mathrm{e}^{\langle\theta,X_{\tau_m}-X_{\tau_1}\rangle-\Lambda_a(\theta)(\tau_m-\tau_1)+p\left(2|\theta|+|s|\right)\left(\tau_{m+1}-\tau_m\right)+p|s|\left(\tau_m-\tau_{m-K}\right)}\right]^{1/p}\nonumber\\
&\quad\quad\times\prod_{i=1}^K E_o\left[\mathrm{e}^{\langle\theta,X_{\tau_m}-X_{\tau_1}\rangle-\Lambda_a(\theta)(\tau_m-\tau_1)+(Kq)s\left(G_i+G_{i+K}+\cdots+G_{i+[\frac{m-K-i-1}{K}]K}\right)}\right]^{1/(Kq)}\nonumber\\
&\quad\leq E_o\left[\exp\{\left(2|\theta|+|s|\right)\tau_1\}\right]E_o\left[\left.\exp\{p\left(2|\theta|+|s|\right)\tau_1\}\,\right|\,\beta=\infty\right]^{\frac{K+1}{p}}\label{cumartesish}\\
&\quad\quad\times E_o\left[\left.\exp\{\langle\theta,X_{\tau_K}\rangle-\Lambda_a(\theta)\tau_K+(Kq)sG_o\}\,\right|\,\beta=\infty\right]^{\frac{m-K-1}{Kq}}\!\!\nonumber.
\end{align}
The last inequality follows from the fact that $(G_i,G_{i+K},\ldots)$ is an i.i.d.\ sequence. The terms of the product in (\ref{cumartesish}) are finite by Theorem \ref{bombagibi} if $p(2|\theta|+|s|)<c_2$ and $2|\theta|+(Kq)|s|<c_2$.  Putting the pieces together,
\begin{align*}
&\limsup_{n\to\infty}\frac{1}{n}\log E_o\left[\exp\{\langle\theta,X_n\rangle-\Lambda_a(\theta)n+ns\langle g,\bar{\nu}_{n,X}^\infty\rangle\}\right]\\
&\qquad\leq0\vee\limsup_{n\to\infty}\frac{1}{n}\log\!\!\!\sum_{m=K+1}^n \!\!E_o\left[\left.\exp\{\langle\theta,X_{\tau_K}\rangle-\Lambda_a(\theta)\tau_K+(Kq)sG_o\}\,\right|\,\beta=\infty\right]^{\frac{m-K-1}{Kq}}\\
&\qquad\leq0\vee\frac{1}{Kq}\log E_o\left[\left.\exp\{\langle\theta,X_{\tau_K}\rangle-\Lambda_a(\theta)\tau_K+(Kq)sG_o\}\,\right|\,\beta=\infty\right].
\end{align*}
Let $h(s):=\frac{1}{Kq}\log E_o\left[\left.\exp\{\langle\theta,X_{\tau_K}\rangle-\Lambda_a(\theta)\tau_K+(Kq)sG_o\}\,\right|\,\beta=\infty\right]$. Lemma \ref{phillysh_one} implies that $h(0)=0$. The map $s\mapsto h(s)$ is analytic in a neighborhood of $0$, and
\begin{align*}
h'(0)&=E_o\left[\left.G_o\,\exp\{\langle\theta,X_{\tau_K}\rangle-\Lambda_a(\theta)\tau_K\}\,\right|\,\beta=\infty\right]\\
&=E_o\left[\left.\sum_{k=0}^{\tau_{1}-1}g((Z_{k+i})_{i\geq1})\,\exp\{\langle\theta,X_{\tau_K}\rangle-\Lambda_a(\theta)\tau_K\}\,\right|\,\beta=\infty\right]\\
&=E_o\left[\left.\left(\sum_{k=0}^{\tau_{1}-1}f((Z_{k+i})_{i\geq1})-\tau_1\int\!\! f\mathrm{d}\bar{\mu}_\xi^\infty\right)\exp\{\langle\theta,X_{\tau_K}\rangle-\Lambda_a(\theta)\tau_K\}\,\right|\,\beta=\infty\right]=0
\end{align*}
by Definition \ref{definemuyuanan}. We conclude that
\begin{equation}\label{yilish}
\limsup_{n\to\infty}\frac{1}{n}\log E_o\left[\exp\{\langle\theta,X_n\rangle-\Lambda_a(\theta)n+ns\langle g,\bar{\nu}_{n,X}^\infty\rangle\}\right]\leq o(s).
\end{equation}

Whenever $s>0$ is small enough, Chebyshev's inequality and the averaged LDP give
\begin{align*}
&\limsup_{n\rightarrow\infty}\frac{1}{n}\log P_o\left(\ \int\!\! f\mathrm{d}\bar{\nu}_{n,X}^\infty-\int\!\! f\mathrm{d}\bar{\mu}_\xi^\infty>\epsilon\ \left|\ |\frac{X_n}{n}-\xi|\leq\delta\right.\right)\\
&\quad=\limsup_{n\rightarrow\infty}\frac{1}{n}\log P_o\left(\langle g,\bar{\nu}_{n,X}^\infty\rangle>\epsilon,|\frac{X_n}{n}-\xi|\leq\delta\right)-\lim_{n\rightarrow\infty}\frac{1}{n}\log P_o\left(|\frac{X_n}{n}-\xi|\leq\delta\right)\\
&\quad\leq\limsup_{n\rightarrow\infty}\frac{1}{n}\log E_o\left[\exp\{\langle\theta,X_n\rangle\},\langle g,\bar{\nu}_{n,X}^\infty\rangle>\epsilon,|\frac{X_n}{n}-\xi|\leq\delta\right]-\langle\theta,\xi\rangle + I_a(\xi) + |\theta|\delta\\
&\quad\leq\limsup_{n\rightarrow\infty}\frac{1}{n}\log E_o\left[\exp\{\langle\theta,X_n\rangle-\Lambda_a(\theta)n\},\langle g,\bar{\nu}_{n,X}^\infty\rangle>\epsilon\right] + |\theta|\delta\\
&\quad\leq\limsup_{n\to\infty}\frac{1}{n}\log E_o\left[\exp\{\langle\theta,X_n\rangle-\Lambda_a(\theta)n+ns\langle g,\bar{\nu}_{n,X}^\infty\rangle\}\right] - s\epsilon + |\theta|\delta\\
&\quad\leq\,o(s) - s\epsilon + |\theta|\delta\\
&\quad\leq-s\epsilon/2+ |\theta|\delta\end{align*} for every $\delta>0$. Similarly, $$\limsup_{n\rightarrow\infty}\frac{1}{n}\log P_o\left(\ \int\!\! f\mathrm{d}\bar{\nu}_{n,X}^\infty-\int\!\! f\mathrm{d}\bar{\mu}_\xi^\infty<-\epsilon\ \left|\ |\frac{X_n}{n}-\xi|\leq\delta\right.\right)\leq-s\epsilon/2+ |\theta|\delta.$$By combining these two bounds, we finally deduce that \[\limsup_{\delta\to0}\limsup_{n\rightarrow\infty}\frac{1}{n}\log P_o\left(\ \left|\int\!\! f\mathrm{d}\bar{\nu}_{n,X}^\infty-\int\!\! f\mathrm{d}\bar{\mu}_\xi^\infty\right|>\epsilon\ \left|\ |\frac{X_n}{n}-\xi|\leq\delta\right.\right)\leq-s\epsilon/2.\qedhere\]
\end{proof}

\begin{proof}[Proof in the nestling case]
Since $\xi\in\mathcal{A}$, there exists a unique $\theta\in\mathcal{C}$ such that $\Lambda_a(\theta)>0$ and $\xi=\nabla\Lambda_a(\theta)$. If $0<s<\Lambda_a(\theta)$, then the first term in (\ref{cokazkaldish}) is bounded from above by $E_o\left[n<\tau_{K+1},\,\exp\{|\theta||X_n|\}\right]$ which goes to $0$ as $n\to\infty$ by (\ref{moment}) and the monotone convergence theorem.

For any $p,q$ with $1<p<c_1/|\theta|$ and $1/p+1/q=1$, each term of the sum in (\ref{cokazkaldish}) can be bounded using H\"older's inequality when $0<s<{\Lambda_a(\theta)}/{(p\vee Kq)}$:
\begin{align}
&E_o\left[\tau_m\leq n<\tau_{m+1},\,\exp\{\langle\theta,X_n\rangle-\Lambda_a(\theta)n+ns\langle g,\bar{\nu}_{n,X}^\infty\rangle\}\right]\nonumber\\
&\quad\leq E_o\left[\mathrm{e}^{\langle\theta,X_{\tau_1}\rangle+\langle\theta,X_{\tau_m}-X_{\tau_1}\rangle-\Lambda_a(\theta)(\tau_m-\tau_1)+s\left(G_1+\cdots+G_{m-1}\right)}\!\!\!\!\!\sup_{\tau_m\leq n<\tau_{m+1}}\!\!\!\!\!\mathrm{e}^{\langle\theta,X_n-X_{\tau_m}\rangle}\right]\nonumber\\
&\quad\leq E_o\left[\exp\{\langle\theta,X_{\tau_1}\rangle\}\right]E_o\left[\mathrm{e}^{\langle\theta,X_{\tau_m}-X_{\tau_1}\rangle-\Lambda_a(\theta)(\tau_m-\tau_1)+ps\left(G_{m-K}+\cdots+G_{m-1}\right)}\!\!\!\!\!\!\!\sup_{\tau_m\leq n<\tau_{m+1}}\!\!\!\!\!\!\!\mathrm{e}^{p\langle\theta,X_n-X_{\tau_m}\rangle}\right]^{1/p}\nonumber\\
&\quad\quad\times\prod_{i=1}^K E_o\left[\mathrm{e}^{\langle\theta,X_{\tau_m}-X_{\tau_1}\rangle-\Lambda_a(\theta)(\tau_m-\tau_1)+(Kq)s\left(G_i+G_{i+K}+\cdots+G_{i+[\frac{m-K-i-1}{K}]K}\right)}\right]^{1/(Kq)}\nonumber\\
&\quad\leq E_o\left[\exp\{\langle\theta,X_{\tau_1}\rangle\}\right]E_o\left[\left.\sup_{\tau_{K}\leq n<\tau_{K+1}}\!\!\!\!\exp\{p|\theta||X_n|\}\,\right|\,\beta=\infty\right]^{1/p}\label{sokush}\\
&\quad\quad\times E_o\left[\left.\exp\{\langle\theta,X_{\tau_K}\rangle-\Lambda_a(\theta)\tau_K+(Kq)sG_o\}\,\right|\,\beta=\infty\right]^{\frac{m-K-1}{Kq}}.\nonumber
\end{align}
The first two terms in (\ref{sokush}) are finite by (\ref{moment}). The last term in (\ref{sokush}) is equal to the last term in (\ref{cumartesish}). The rest of the argument is identical to the one given in the non-nestling case.
\end{proof}

\begin{proof}[Proof of Theorem \ref{sevval}]
Fix $\xi\in\mathcal{A}$. Take any $\alpha\in\mathcal{E}$ with $m(\alpha)=\xi$. The corresponding transient process $Q_\alpha$ induces a probability measure $\bar{\alpha}$ on $U^{\mathbb{N}}$ via the map $(x_i)_{i\in\mathbb{Z}}\mapsto(x_1-x_o,x_2-x_1,\ldots)$. If $\bar{\alpha}\neq\bar{\mu}_\xi^\infty$, then there exist $K\in\mathbb{N}$, $f:U^{\mathbb{N}}\rightarrow\mathbb{R}$ and $\epsilon>0$ such that $f((z_i)_{i\geq1})$ is bounded and independent of $(z_i)_{i>K}$, and $|\langle f,\bar{\alpha} - \bar{\mu}_\xi^\infty\rangle|>\epsilon$.

For every $w\in W_{\infty}^{tr}$, $m\in\mathbb{N}$, and $(x_1,x_2,\ldots,x_m)$ such that $(x_{i+1}-x_i)\in U$, it follows easily from (\ref{ellipticity}) that $$P_o\left(X_1=x_1,\ldots,X_m=x_m\right)\geq\kappa^LQ^w\left(X_1=x_1,\ldots,X_m=x_m\right)$$ if $(x_1,x_2,\ldots,x_m)$ intersects $w$ at most $L$ times. With this observation in mind, let $H(n,X)$ denote the number of times $(X_1,\ldots,X_n)$ intersects $(X_i)_{i\leq0}$. Since the walk under $Q_\alpha$ is transient in the $\xi$ direction, there exists a constant $L$ such that $\lim_{n\to\infty}Q_\alpha(H(n,X)\leq L)\geq1/2$. For notational convenience, let $$A_n^\delta:=\left\{|\langle f,\bar{\nu}_{n,X}^\infty - \bar{\mu}_\xi^\infty\rangle|>\epsilon,\,|\frac{X_n}{n}-\xi|\leq\delta,\,H(n+K,X)\leq L\right\}.$$ 

By Jensen's inequality,
\begin{align*}
&P_o\left(|\langle f,\bar{\nu}_{n,X}^\infty - \bar{\mu}_\xi^\infty\rangle|>\epsilon,\,|\frac{X_n}{n}-\xi|\leq\delta\right)\\
&\qquad\geq\kappa^L\sup_{w\in W_\infty^{\mathrm{tr}}}Q^w\left(A_n^\delta\right)\\
&\qquad\geq\kappa^L\int E^w\left[\one_{A_n^\delta}\right]\,\mathrm{d}\alpha(w)\\
&\qquad=\kappa^L\int E_\alpha^w\left[\one_{A_n^\delta}\,\left.\frac{\mathrm{d}Q^w}{\mathrm{d}Q_\alpha^w}\right|_{\sigma(Z_1,\ldots,Z_{n+K})}\right]\,\mathrm{d}\alpha(w)\\
&\qquad=\kappa^LQ_\alpha(A_n^\delta)\frac{1}{Q_\alpha(A_n^\delta)}\int_{A_n^\delta}\exp\left(-\log\frac{\mathrm{d}Q_\alpha^w}{\mathrm{d}Q^w}(z_1,\ldots,z_{n+K})\right)\mathrm{d}Q_\alpha(w,z_1,\ldots,z_{n+K})\\
&\qquad\geq\kappa^LQ_\alpha(A_n^\delta)\exp\left(-\frac{1}{Q_\alpha(A_n^\delta)}\int_{A_n^\delta}\!\!\!\log\frac{\mathrm{d}Q_\alpha^w}{\mathrm{d}Q^w}(z_1,\ldots,z_{n+K})\mathrm{d}Q_\alpha(w,z_1,\ldots,z_{n+K})\right).
\end{align*}Since $m(\alpha)=\xi$ and $|\langle f,\bar{\alpha} - \bar{\mu}_\xi^\infty\rangle|>\epsilon$, the $L^1$-ergodic theorem implies that $$\lim_{n\to\infty}Q_\alpha(A_n^\delta)=\lim_{n\to\infty}Q_\alpha(H(n+K,X)\leq L)\geq1/2.$$
Therefore,
\begin{align*}
&\liminf_{n\to\infty}\frac{1}{n}\log P_o\left(|\langle f,\bar{\nu}_{n,X}^\infty - \bar{\mu}_{\xi}^\infty\rangle|>\epsilon,\,|\frac{X_n}{n}-\xi|\leq\delta\right)\\
&\quad\geq-\limsup_{n\to\infty}\frac{1}{nQ_\alpha(A_n^\delta)}\int_{A_n^\delta}\log\frac{\mathrm{d}Q_\alpha^w}{\mathrm{d}Q^w}(z_1,\ldots,z_{n+K})\mathrm{d}Q_\alpha(w,z_1,\ldots,z_{n+K})\\
&\quad=-\int_{W_\infty^{\mathrm{tr}}}\left[\sum_{z\in U}q_\alpha(w,z)\log\frac{q_\alpha(w,z)}{q(w,z)}\right]\mathrm{d}\alpha(w)=-\mathfrak{I}_a(\alpha)
\end{align*} again by the $L^1$-ergodic theorem. Finally, Theorem \ref{averagedconditioningsh} and the averaged LDP give
\begin{align*}
0>&\limsup_{\delta\to0}\limsup_{n\rightarrow\infty}\frac{1}{n}\log P_o\left(\ \left|\int\!\! f\mathrm{d}\bar{\nu}_{n,X}^\infty-\int\!\! f\mathrm{d}\bar{\mu}_\xi^\infty\right|>\epsilon\ \left|\ |\frac{X_n}{n}-\xi|\leq\delta\right.\right)\\
=&\,I_a(\xi)+\limsup_{\delta\to0}\limsup_{n\rightarrow\infty}\frac{1}{n}\log P_o\left(\ \left|\int\!\! f\mathrm{d}\bar{\nu}_{n,X}^\infty-\int\!\! f\mathrm{d}\bar{\mu}_\xi^\infty\right|>\epsilon,\,|\frac{X_n}{n}-\xi|\leq\delta\right)\\
\geq&\,I_a(\xi)-\mathfrak{I}_a(\alpha).
\end{align*} In words, $\alpha$ is not the minimizer of (\ref{divaneasik}). Theorem \ref{qual} and Lemma \ref{tugish} imply that the infimum in (\ref{divaneasik}) is attained. Therefore, the probability measure that any minimizer of (\ref{divaneasik}) induces on $U^{\mathbb{N}}$ is equal to $\bar{\mu}_{\xi}^\infty$. This implies that $\bar{\mu}_{\xi}^\infty$ corresponds to a transient process with stationary and ergodic increments, and $\mu_{\xi}^\infty$ (which is defined in the statement of Theorem \ref{sevval}) is the unique minimizer of (\ref{divaneasik}).
\end{proof}
\begin{remark}
The argument above indirectly proves that $\mu_\xi^\infty\in\mathcal{E}$, and that $m(\mu_\xi^\infty)=\xi$. These facts are also easy to show directly using Definition \ref{definemuyuanan}. In fact, $\mu_\xi^\infty$ is mixing with rate given by the tail behaviour of $\tau_1$.
\end{remark}

\section*{Acknowledgments}

I thank S.\ R.\ S.\ Varadhan for valuable discussions. I also thank J.\ Peterson, O.\ Zeitouni and an anonymous referee for pointing out that all of the results and proofs in this paper are valid under condition (\textbf{T}) which is strictly weaker than what I was assuming in an earlier version, namely Kalikow's condition.


\begin{thebibliography}{9}

\bibitem {Berger}
\textsc{Berger, N.} (2008).
Limiting velocity of high-dimensional random walk in random environment. \textit{Ann. Probab.}
\textbf{36} 728--738.

\bibitem {CGZ}
\textsc{Comets, F.}, \textsc{Gantert, N.} and \textsc{Zeitouni, O.} (2000).
Quenched, annealed and functional large deviations for one dimensional random walk in random environment. \textit{Probab. Theory Related Fields.}
\textbf{118} 65--114.

\bibitem {DeMasi}
\textsc{De Masi, A.}, \textsc{Ferrari, P. A.}, \textsc{Goldstein, S.} and \textsc{Wick, W. D.} (1989).
An invariance principle for reversible Markov processes with applications to random motions in random environments. \textit{J. Stat. Phys.}
\textbf{55} 787--855.

\bibitem {DZ}
\textsc{Dembo, A.} and \textsc{Zeitouni, O.} (1998).
\textit{Large deviation techniques and applications}, 2nd ed.
Springer, New York.

\bibitem {DV4}
\textsc{Donsker, M. D.} and \textsc{Varadhan, S. R. S.} (1983).
Asymptotic evaluation of certain Markov process expectations for large time. IV. \textit{Comm. Pure Appl. Math.}
\textbf{36} 183--212.

\bibitem {GdH}
\textsc{Greven, A.} and \textsc{den Hollander, F.} (1994).
Large deviations for a random walk in random environment. \textit{Ann. Probab.}
\textbf{22} 1381--1428.

\bibitem {Kalikow}
\textsc{Kalikow, S. A.} (1981).
Generalized random walk in a random environment. \textit{Ann. Probab.}
\textbf{9} 753--768.

\bibitem {KV}
\textsc{Kipnis, C.} and \textsc{Varadhan, S. R. S.} (1986).
A central limit theorem for additive functionals of reversible Markov processes and applications to simple exclusion. \textit{Comm. Math. Phys.}
\textbf{104} 1--19.


\bibitem {Kozlov}
\textsc{Kozlov, S. M.} (1985).
The averaging method and walks in inhomogeneous environments. \textit{Russian Math. Surveys (Uspekhi Mat. Nauk)}
\textbf{40} 73--145.

\bibitem {implicit}
\textsc{Krantz, S. G.} and \textsc{Parks, H. R.} (2002).
\textit{The implicit function theorem: history, theory, and applications}.
Birkh\"auser, Boston.


\bibitem {Olla}
\textsc{Olla, S.} (1994).
\textit{Homogenization of diffusion processes in random fields}. Ecole Polytecnique, Palaiseau.

\bibitem {PV}
\textsc{Papanicolaou, G.} and \textsc{Varadhan, S. R. S.} (1981).
\textit{Boundary value problems with rapidly oscillating random coefficients} in "Random Fields", J. Fritz, D. Szasz editors, Janyos Bolyai series.
North-Holland, Amsterdam.

\bibitem {PetersonThesis}
\textsc{Peterson, J.} (2008).
Limiting distributions and large deviations for random walks in random environments. Ph.D. thesis, University of Minnesota.

\bibitem {JonOfer}
\textsc{Peterson, J.} and \textsc{Zeitouni, O.} (2008).
On the annealed large deviation rate function for a multi-dimensional random walk in random environment. Preprint.
arXiv:0812.3619


\bibitem {FirasLDP}
\textsc{Rassoul-Agha, F.} (2004).
Large deviations for random walks in a mixing random environment and other (non-Markov) random walks. \textit{Comm. Pure Appl. Math.}
\textbf{57} 1178--1196.

\bibitem {Rockafellar}
\textsc{Rockafellar, T.} (1972).
\textit{Convex analysis}, 2nd ed.
Princeton University, New Jersey.

\bibitem {jeffrey}
\textsc{Rosenbluth, J.} (2006).
Quenched large deviations for multidimensional random walk in random environment: A variational formula. Ph.D.\ thesis, Courant Institute, New York University.
arXiv:0804.1444

\bibitem {SznitmanZerner}
\textsc{Sznitman, A. S.} and \textsc{Zerner, M.} (1999).
A law of large numbers for random walks in random environment. \textit{Ann. Probab.}
\textbf{27} 1851--1869.

\bibitem {SznitmanSlowdown}
\textsc{Sznitman, A. S.} (2000).
Slowdown estimates and central limit theorem for random walks in random environment. \textit{J. Eur. Math. Soc.}
\textbf{2} 93--143.

\bibitem {SznitmanT}
\textsc{Sznitman, A. S.} (2001).
On a class of transient random walks in random environment. \textit{Ann. Probab.}
\textbf{29} 724--765.

\bibitem {Sznitman}
\textsc{Sznitman, A. S.} (2002).
\textit{Lectures on random motions in random media} in "Ten Lectures on Random Media", DMV-Lectures \textbf{32}.
Birkh\"{a}user, Basel.


\bibitem {Raghu}
\textsc{Varadhan, S. R. S.} (2003).
Large deviations for random walks in a random environment. \textit{Comm. Pure Appl. Math.}
\textbf{56} 1222--1245.

\bibitem {YilmazThesis}
\textsc{Yilmaz, A.} (2008).
Large deviations for random walk in a random environment. Ph.D.\ thesis, Courant Institute, New York University.
arXiv:0809.1227

\bibitem {YilmazQuenched}
\textsc{Yilmaz, A.} (2009).
Quenched large deviations for random walk in a random environment. \textit{Comm. Pure Appl. Math.}
\textbf{62} 1033--1075.

\bibitem {YilmazQA}
\textsc{Yilmaz, A.} (2009).
On the equality of the quenched and averaged large deviation rate functions for high-dimensional ballistic random walk in a random environment. Preprint.
arXiv:0903.0410

\bibitem {Zeitouni}
\textsc{Zeitouni, O.} (2006).
Random walks in random environments. \textit{J. Phys. A: Math. Gen.}
\textbf{39} R433--464.

\bibitem {Zerner}
\textsc{Zerner, M. P. W.} (1998).
Lyapounov exponents and quenched large deviations for multidimensional random walk in random environment. \textit{Ann. Probab.}
\textbf{26} 1446--76.

\end{thebibliography}
\end{document}